\documentclass[12pt,reqno]{amsart}
\usepackage{amssymb,a4wide,eucal,amscd,yhmath,amsgen,amsopn}
\usepackage{fancyhdr}
\usepackage{amsmath,hyperref,xcolor}
\usepackage{mathrsfs}
\usepackage[all,cmtip]{xy}
\SilentMatrices
\SelectTips{eu}{}
\newtheorem{theorem}{Theorem}[section]

\newtheorem{lemma}[theorem]{Lemma}
\newtheorem{Not}[theorem]{Notation}
\theoremstyle{rem}

\theoremstyle{definition}
\newtheorem{definition}[theorem]{Definition}
\theoremstyle{construct}
\newtheorem{construct}[theorem]{Construction}
\theoremstyle{examp}

\newtheorem{Fact}[theorem]{Fact}
\newcommand\projective\mathbf
\newcommand\PP{\projective P}

\newcommand\OO{\mathcal O}

\newcommand\ZZ{\mathbb Z}

\newcommand\GG{\mathbb Gr}

\newcommand\onto\twoheadrightarrow
\newcommand\lra\longrightarrow
\newcommand\dar\downarrow
\DeclareMathOperator{\pic}{Pic}
\DeclareMathOperator{\im}{im}
\DeclareMathOperator{\cok}{coker}
\DeclareMathOperator{\rk}{rank}
\DeclareMathOperator{\Hom}{Hom}

\begin{document}

\title{Indecomposable bundles on Cartesian products of odd projective spaces}
\author{Damian M Maingi}
\date{April,2025}

\address{Department of Mathematics\\Catholic University of Eastern Africa\\P.O Box 62157, 00200 Nairobi}
\email{dmaingi@cuea.edu}

\maketitle

\begin{abstract}
In this paper we construct indecomposable vector bundles associated to monads on Cartesian products of odd dimension projective spaces.
Specifically we establish the existence of monads on $(\PP^1)^{l_1}\times\cdots\times(\PP^{2n+1})^{l_m}$.
We prove stability of the kernel bundle and prove that the cohomology bundle is simple. We also prove the same for monads on
$(\PP^n)^2\times(\PP^m)^2\times(\PP^l)^2$ for an ample line bundle $\mathscr{L}=\mathcal{O}_X(\alpha,\alpha,\beta,\beta,\gamma,\gamma)$.
\end{abstract}

\section{Introduction}
\noindent The existence of indecomposable low rank vector bundles on algebraic varieties in comparison with the dimension of the ambient space
has been a fertile area in algebraic geometry for the last 50 years. Nonetheless, it remains intriguing, fascinating and exciting to construct new examples of indecomposable low rank vector bundles.
Remarkable works in  this regard are: the famous Horrocks-Mumford bundle of rank 2 over $\PP^4$ \cite{6}, the Horrocks vector bundle of rank 3 on $\PP^5$ \cite{4},
other examples were given by Tango in \cite{16} and \cite{17} and all these were obtained as cohomologies of certain monads.\\
\\

\noindent Monads appear in many contexts within algebraic geometry and were first introduced by Horrocks\cite{5} where he proved that all vector bundles $E$ on $\PP^3$ could be obtained as the cohomology bundle of a given monad.
The goal of this paper is the construction of simple vector bundles associated to monads on Cartesian products of projective spaces. 
Fl\o{}ystad\cite{3} established the existence of monads on projective spaces $\PP^k$. Marchesi et al \cite{13} extended this further for more generalized projective varieties. 
Costa and Miro\cite{2} established existence of monads on smooth hyperquadrics.\\

\noindent Maingi in \cite{8} constructed bundles associated to monads on $\PP^{n}\times\PP^{m}$ of the form
\[\begin{CD}0@>>>\OO_{X}(-\rho,-\sigma)^{\oplus{\alpha}}@>>>{\OO^{\oplus{\beta}}_X} @>>>\OO_{X}(\rho,\sigma)^{\oplus{\gamma}}  @>>>0,\end{CD}\]
  Maingi in \cite{9} constructed bundles associated to monads on $\PP^{2n+1}\times\PP^{2n+1}$ of the form
\[\begin{CD}0@>>>\OO_{X}(-1,-1)^{\oplus{k}}@>>^{f}>{\OO^{\oplus{2n}\oplus{2k}}_X} @>>^{g}>\OO_{X}(1,1)^{\oplus{k}}  @>>>0,\end{CD}\]
 Maingi in \cite{11} established existence of monads $\PP^{n}\times\PP^{n}\times\PP^{m}\times\PP^{m}$ of the form
\[\begin{CD}0@>>>{\OO_X(-1,-1,-1,-1)^{\oplus k}} @>>>{\mathscr{G}_{n}\oplus\mathscr{G}_{m}}@>>>\OO_X(1,1,1,1)^{\oplus k} @>>>0.\end{CD}\]
He generalized these results in \cite{10} i.e. he established the existence of monads
\[\begin{CD}0@>>>{\OO_X(-1,\cdots,-1)^{\oplus k}} @>>^{f}>{\mathscr{G}_1\oplus\cdots\oplus\mathscr{G}_n}@>>^{g}>\OO_X(1,\cdots,1)^{\oplus k} @>>>0\end{CD}\]
on $X=\PP^{a_1}\times\PP^{a_1}\times\PP^{a_2}\times\PP^{a_2}\times\cdots\times\PP^{a_n}\times\PP^{a_n}$
where 
\begin{align*}
\mathscr{G}_1:=\OO_X(-1,0,0,\cdots,0)^{\oplus a_1+\oplus k}\oplus\OO_X(0,-1,0,0,\cdots,0)^{\oplus a_1+\oplus k}\\
\mathscr{G}_2:=\OO_X(0,0,-1,\cdots,0)^{\oplus a_2+\oplus k}\oplus\OO_X(0,0,0,-1,\cdots,0)^{\oplus a_2+\oplus k}\\
\cdots\cdots\cdots\cdots\cdots\cdots\cdots\cdots\cdots\cdots\cdots\cdots\cdots\cdots\cdots\cdots\cdots\cdots\cdots\\
\mathscr{G}_n:=\OO_X(0,0,\cdots,0,-1,0)^{\oplus a_n+\oplus k}\oplus\OO_X(0,0,\cdots,0,-1)^{\oplus a_n+\oplus k}.
\end{align*}
Part of the results in this paper generalize the results of the two immediately above mentioned papers in that the polarisation is $\mathscr{L}=\OO_X(\alpha,\alpha,\beta,\beta,\gamma,\gamma)$.
More recently Maingi in \cite{12} established existence of monads
\[
\begin{CD}
0@>>>\OO_{X}(-1,\cdots,-1)^{\oplus{k}}@>>^{f}>{\OO^{\oplus{2\mu}\oplus{2k}}_X} @>>^{g}>\OO_{X}(1,\cdots,1)^{\oplus{k}}  @>>>0\\
\end{CD}
\]
on a Cartesian product $X = (\mathbb{P}^1)^l\times(\mathbb{P}^3)^m\times(\mathbb{P}^5)^n$ of projective spaces,
where $l,m,n,k$ are positive integers and $\mu=2^{l+2m+n-1}3^n-1$. 
In this paper we extend these results to a Cartesian product of odd projective spaces $(\PP^1)^{l_1}\times(\PP^3)^{l_2}\times\cdots\times(\PP^{2n+1})^{l_m}$.
\vspace{0.5cm}

\noindent In this work we give generalizations for previous results by several authors. To be specific we build upon results by Maingi \cite{10,11,12} therefore the definitions, notation, the methods applied are quite similar and the trend follows the paper by 
Ancona and Ottaviani \cite{1}. \\

\noindent The main results in this paper are:

\begin{theorem}
 Let $X = (\PP^1)^{l_1}\times(\PP^3)^{l_2}\times\cdots\times(\PP^{2n+1})^{l_m}$,  be a Cartesian product of $l_1$ copies 
of $\PP^1$, $l_2$ copies of $\PP^3$, $\cdots$ and $l_m$ copies of $\PP^{2n+1}$. There exists a monad of the form
 \[
\begin{CD}
M_\bullet: 0@>>>\OO_{X}(-1,\cdots,-1)^{\oplus{k}}@>>^{\overline{A}}>{\OO^{\oplus{2\nu}\oplus{2k}}_X} @>>^{\overline{B}}>\OO_{X}(1,\cdots,1)^{\oplus{k}}  @>>>0\\
\end{CD}
\]
where $l_1,\cdots,l_m,\nu,k$ are positive integers and $\nu=2^{l_1-1}4^l_2\cdots(2n+2)^l_m-1$.

\begin{enumerate}
 \item The kernel of ${\overline{B}}$, $\ker({\overline{B}})$ is stable and
 \item The cohomology bundle $E=\ker{\overline{B}}/\im{\overline{A}}$ is indecomposable.
\end{enumerate}

\end{theorem}

\begin{theorem}
Let $\alpha,\beta,\gamma,l,m,n$ and $k$ be nonnegative integers. Then there exists a linear monad on $X = (\PP^{n})^2\times(\PP^{m})^2\times(\PP^{l})^2$ of the form;
\[\begin{CD}0\rightarrow{\OO_X(-\alpha,-\alpha,-\beta,-\beta,-\gamma,-\gamma)^{\oplus k}} @>>^{f}>{\mathscr{G}_{\alpha}\oplus\mathscr{G}_{\beta}\oplus\mathscr{G}_{\gamma}}@>>^{g}>\OO_X(\alpha,\alpha,\beta,\beta,\gamma,\gamma)^{\oplus k}\rightarrow0\end{CD}\]
where 
\begin{align*}
\mathscr{G}_{\alpha}:=\OO_X(-\alpha,0,0,0,0,0)^{\oplus n+\oplus k}\oplus\OO_X(0,-\alpha,0,0,0,0)^{\oplus n+\oplus k}\\
\mathscr{G}_{\beta}:=\OO_X(0,0,-\beta,0,0,0)^{\oplus m+\oplus k}\oplus\OO_X(0,0,0,-\beta,0,0)^{\oplus m+\oplus k}\\
\mathscr{G}_{\gamma}:=\OO_X(0,0,0,0,-\gamma,0)^{\oplus l+\oplus k}\oplus\OO_X(0,0,0,0,0,-\gamma)^{\oplus l+\oplus k}
\end{align*}
with the properties

\begin{enumerate}
 \item The kernel of $g$, $\ker(g)$ is stable and
 \item The cohomology bundle $E$ is indecomposable.
\end{enumerate}

\end{theorem}

\section{Preliminaries}

\noindent In this section we define and give notation in order to set up for the main results.
Most of the definitions are from chapter two of the book by Okonek, Schneider and Spindler \cite{14}. In this paper we will work over an algebraically closed field of characteristic zero.

\begin{definition}
Let $X$ be a nonsingular projective variety. 
\begin{enumerate}
\renewcommand{\theenumi}{\alph{enumi}}
 \item A {\it{monad}} on $X$ is a complex of vector bundles:
\[\xymatrix{0\ar[r] & M_0 \ar[r]^{\alpha} & M_1 \ar[r]^{\beta} & M_2 \ar[r] & 0}\]
exact at $M_0$ and at $M_2$ i.e. $\alpha$ is injective and $\beta$ surjective.
\item The image of $\alpha$ is a subbundle of $B$ and the bundle $E = \ker(\beta)/\im (\alpha)$ and is called the cohomology bundle of the monad.
\end{enumerate}
\end{definition}

\begin{definition}
A monad \[\xymatrix{0\ar[r] & M_0 \ar[r]^{\alpha} & M_1 \ar[r]^{\beta} & M_2 \ar[r] & 0}\] has a display diagram of short exact sequences as shown below:
\[
\begin{CD}
@.@.0@.0\\
@.@.@VVV@VVV\\
0@>>>{M_0} @>>>\ker{\beta}@>>>E@>>>0\\
@.||@.@VVV@VVV\\
0@>>>{M_0} @>>^{\alpha}>{M_1}@>>>\cok{\alpha}@>>>0\\
@.@.@V^{\beta}VV@VVV\\
@.@.{M_2}@={M_2}\\
@.@.@VVV@VVV\\
@.@.0@.0
\end{CD}
\]
\end{definition}

\begin{definition}
Let $X$ be a nonsingular projective variety, let $\mathscr{L}$ be a very ample line sheaf, and $V,W,U$ be finite dimensional $k$-vector spaces.
A linear monad on $X$ is a complex of sheaves,
\[ M_\bullet:
\xymatrix
{
0\ar[r] & V\otimes {\mathscr{L}}^{-1} \ar[r]^{A} & W\otimes \OO_X \ar[r]^{B} & U\otimes \mathscr{L} \ar[r] & 0
}
\]
where $A\in \Hom(V,W)\otimes H^0 \mathscr{L}$ is injective and $B\in \Hom(W,U)\otimes H^0 \mathscr{L}$ is surjective.\\
The existence of the monad $M_\bullet$ is equivalent to: $A$ and $B$ being of maximal rank and $BA$ being the zero matrix.
\end{definition}

\begin{definition}
Let $X$ be a non-singular irreducible projective variety of dimension $d$ and let $\mathscr{L}$ be an ample line bundle on $X$. For a 
torsion-free sheaf $F$ on $X$ we define
\begin{enumerate}
\renewcommand{\theenumi}{\alph{enumi}}
 \item the degree of $F$ relative to $\mathscr{L}$ as $\deg_{\mathscr{L}}F:= c_1(F)\cdot \mathscr{L}^{d-1}$, where $c_1(F)$ is the first Chern class of $F$
 \item the slope of $F$ as $\nu_{\mathscr{L}}(F):= \frac{\deg_{\mathscr{L}}F}{rk(F)}$.
  \end{enumerate}
\end{definition}

\subsection{Hoppe's Criterion over polycyclic varieties.}
Suppose that the Picard group Pic$(X) \simeq \ZZ^l$ where $l\geq2$ is an integer then $X$ is a polycyclic variety.
Given a divisor $B$ on $X$ we define $\delta_{\mathscr{L}}(B):= \deg_{\mathscr{L}}\OO_{X}(B)$.
Then one has the following stability criterion ({\cite{7}, Theorem 3}):

\begin{theorem}[Generalized Hoppe Criterion]
 Let $G\rightarrow X$ be a holomorphic vector bundle of rank $r\geq2$ over a polycyclic variety $X$ equipped with a polarisation 
 $\mathscr{L}$.
 \\
 If \[H^0(X,(\wedge^sG)\otimes\OO_X(B))=0\] 
 for all $B\in\pic(X)$ and $s\in\{1,\ldots,r-1\}$ such that
 $\begin{CD}\displaystyle{\delta_{\mathscr{L}}(B)<-s\nu_{\mathscr{L}}(G)}\end{CD}$ then $G$ is stable and if
 $\begin{CD}\displaystyle{\delta_{\mathscr{L}}(B)\leq-s\nu_{\mathscr{L}}(G)}\end{CD}$ then $G$ is semi-stable.\\
\\
 Conversely if then $G$ is (semi-)stable then  \[H^0(X,G\otimes\OO_X(B))=0\]
 for all $B\in\pic(X)$ such that $\left(\delta_{\mathscr{L}}(B)\leq\right)$ $\delta_{\mathscr{L}}(B)<-\nu_{\mathscr{L}}(G)$.
\end{theorem}

\begin{Not}

\noindent In section 3 the ambient space is $X=(\PP^1)^{l_1}\times(\PP^3)^{l_2}\times\cdots\times(\PP^{2n+1})^{l_m}$ then $\pic(X) \simeq \ZZ^l$, where
$\displaystyle{l=\sum_{i=1}^m{l_i}}$.
\\
Let $p_{11},\cdots,p_{1l_1}$ be natural projections  from $X$ onto $\PP^1$, 
$p_{21},\cdots,p_{2l_2}$ be natural projections  from $X$ onto $\PP^3$,
$\cdots\cdots\cdots$
$p_{m1},\cdots,p_{ml_m}$ be natural projections  from $X$ onto $\PP^{2n+1}$.\\

\noindent We shall denote by \\
$g_{1i}$, the generators of $\pic(X)$ corresponding to $p_{1i}^*\OO_{\PP^{1}}(1)$, for $i=1,\cdots,l_1$,\\
$g_{2j}$, the generators of $\pic(X)$ corresponding to $p_{2j}^*\OO_{\PP^{3}}(1)$, for $j=1,\cdots,l_2$,\\
$\cdots\cdots\cdots\cdots\cdots\cdots\cdots\cdots\cdots\cdots\cdots\cdots\cdots\cdots\cdots\cdots\cdots\cdots\cdots\cdots\cdots\cdots\cdots$ and\\
$g_{mk}$,the generators of $\pic(X)$ corresponding to $p_{mk}^*\OO_{\PP^{2n+1}}(1)$, for $k=1,\cdots,l_m$\\
and so $\pic(X)= \left\langle g_{11},\cdots,g_{ml_m}\right\rangle$ since $\pic(X) \simeq \ZZ^{l}$ with $\displaystyle{l=\sum_{i=1}^m{l_i}}$.\\

\noindent Next, we denote by
$\OO_X(g_{11},\cdots,g_{1l_1},g_{21}\cdots,g_{2l_2},\cdots g_{m1},\cdots,g_{ml_m}):= 
{p_{11}}^*\OO_{\PP^1}(g_{11})\otimes\cdots\otimes {p_{1l_1}}^*\OO_{\PP^1}(g_{1l_1})\otimes
{p_{21}}^*\OO_{\PP^3}(g_{21})\otimes\cdots\otimes {p_{2l_2}}^*\OO_{\PP^3}(g_{2l_2})\otimes\cdots\otimes
{p_{m1}}^*\OO_{\PP^{2n+1}}(g_{m1})\otimes\cdots\otimes {p_{ml_m}}^*\OO_{\PP^{2n+1}}(g_{ml_m})$\\

\noindent Suppose $h_{11},\cdots,h_{1l_1}$ are hyperplanes in $\PP^{1}$,
$h_{21},\cdots,h_{2l_2}$ are hyperplanes in $\PP^{3}$,
$\cdots\cdots\cdots$ and
$h_{m1},\cdots,h_{ml_m}$ are hyperplanes in $\PP^{2n+1}$,
with the intersection product induced by
$h^1_{11}=\cdots=h^1_{1l_1}= h^3_{21}=\cdots=h^3_{2l_2}= h^{2n+1}_{m1}=\cdots=h^{2n+1}_{ml_m}= 1$,\\
$h^2_{11}=\cdots=h^2_{1l_1}= h^4_{21}=\cdots=h^4_{2l_2}= h^{2n+2}_{m1}=\cdots=h^{2n+2}_{ml_m}= 0$,\\

\noindent For any line bundle $\mathscr{L} = \OO_X(g_{11},\cdots,g_{ml_m})$ on $X$ and a vector bundle $E$, we write 
$E(g_{11},\cdots,g_{ml_m}) = E\otimes\OO_X(g_{11},\cdots,g_{ml_m})$ 
and $(g_{11},\cdots,g_{ml_m}):= \OO_X(g_{11}h_{11}+\cdots+g_{ml_m}h_{ml_m})$ representing its corresponding divisor.\\

\noindent The ambient space in section 4 is the Cartesian product $X=(\PP^{n})^2\times(\PP^{m})^2\times(\PP^{l})^2$ and so then $\pic(X) \simeq \ZZ^{6}$.
\\
Suppose $\pi_{1n}$ and $\pi_{2n}$ are natural projections  from $X$ onto $\PP^{ n}$,\\
$\pi_{1m}$ and $\pi_{2m}$ are natural projections  from $X$ onto $\PP^{m}$ and \\
$\pi_{1l}$ and $\pi_{2l}$ are natural projections  from $X$ onto $\PP^{l}$.
\\
For all $i=1,2$, we denote by $h_{in}$ the generator of $\pic(X)$ corresponding to $\pi^*_{in}\OO_{\PP^n}(1)$,\\
$h_{im}$ the generator of $\pic(X)$ corresponding to $\pi^*_{im}\OO_{\PP^m}(1)$ and\\
$h_{il}$ the generator of $\pic(X)$ corresponding to $\pi^*_{il}\OO_{\PP^l}(1)$ and so $\pic(X)= \left\langle h_{1n},h_{2n},h_{1m},h_{2m},h_{1l},h_{2l}\right\rangle$.
\\
Denote by $\OO_X(h_{1n},h_{2n},h_{1m},h_{2m},h_{1l},h_{2l}):= \pi_{1n}^*\OO_{\PP^{n}}(h_{1n})\otimes \pi_{2n}^*\OO_{\PP^{n}}(h_{2n})\otimes {\pi_{1m}}^*\OO_{\PP^{m}}(h_{1m})\otimes {\pi_{2m}}^*\OO_{\PP^{m}}(h_{2m})\otimes {\pi_{1l}}^*\OO_{\PP^{l}}(h_{1l})\otimes \pi_{2l}^*\OO_{\PP^{l}}(h_{2l})$,
\\
Suppose $g_{1n}$ and $g_{2n}$ are hyperplanes in $\PP^{n}$,
$g_{1m}$ and $g_{2m}$ are hyperplanes in $\PP^{m}$ and
$g_{1l}$ and $g_{2l}$ are hyperplanes in $\PP^{l}$ with the intersection product induced by
$g_{in}^{n} = g_{im}^{m} = g_{il}^{l} =1$ and $g_{in}^{n+1} = g_{im}^{m+1} = g_{il}^{l+1} = 0$.
\\

\noindent The normalization of $E$ on $X$ with respect to $\mathscr{L}$ is defined as follows:\\
Set $d=\deg_{\mathscr{L}}(\OO_X(1,0,\cdots,0))$, since $\deg_{\mathscr{L}}(E(-k_E,0,\cdots,0))=\deg_{\mathscr{L}}(E)-nk\cdot \rk(E)$ 
there is a unique integer $k_E:=\lceil\nu_\mathscr{L}(E)/d\rceil$ such that  $1 - d.\rk(E)\leq \deg_\mathscr{L}(E(-k_E,0,\cdots,0))\leq0$. 
The twisted bundle $E_{{\mathscr{L}}-norm}:= E(-k_E,0,\cdots,0)$ is called the $\mathscr{L}$-normalization of $E$.

\end{Not}

\noindent The following lemma is actually a corollary of theorem 2.5 above, a special case of the generalized Hoppe criterion on stability.

\begin{lemma}
Let $X$ be a polycyclic variety with Picard number $n$, let $\mathscr{L}$ be an ample line bundle and
let E be a rank $r>1 $ vector bundle over $X$.
If $H^0(X,(\bigwedge^q E)_{{\mathscr{L}}-norm}(p_1,\cdots,p_{n})) = 0$ for $1\leq q \leq r-1$ and every 
$(p_1,\cdots,p_{n})\in \mathbb{Z}^{n}$ such that $\delta_{\mathcal L}(B)\leq0$, where $B:={\mathcal O}_X(p_1,..., p_n)$
then E is $\mathscr{L}$-stable.
\end{lemma}

\begin{Fact}
Let $0\rightarrow E \rightarrow F \rightarrow G\rightarrow0$ be an exact sequence of vector bundles.
Then we have the following exact sequence involving exterior and symmetric powers
\[0\lra\bigwedge^q E \lra\bigwedge^q F \lra\bigwedge^{q-1} F\otimes G\lra\cdots \lra F\otimes S^{q-1}G \lra S^{q}G\lra0\]
\end{Fact}

\begin{theorem}[K\"{u}nneth formula]
 Let $X$ and $Y$ be projective varieties over a field $k$. 
 Let $\mathscr{F}$ and $\mathscr{G}$ be coherent sheaves on $X$ and $Y$ respectively.
 Let $\mathscr{F}\boxtimes\mathscr{G}$ denote $p_1^*(\mathscr{F})\otimes p_2^*(\mathscr{G})$\\
 then $\displaystyle{H^m(X\times Y,\mathscr{F}\boxtimes\mathscr{G}) \cong \bigoplus_{p+q=m} H^p(X,\mathscr{F})\otimes H^q(Y,\mathscr{G})}$.
\end{theorem}

\begin{theorem}[\cite{15}, Theorem 4.1, page 131]
 Let $n\geq1$ be an integer  and $d$ be an integer. We denote by $S_d$ the space of homogeneous polynomials of degree $d$ in 
 $n+1$ variables (conventionally if $d<0$ then $S_d=0$). Then the following statements are true:
 \begin{enumerate}
 \renewcommand{\theenumi}{\alph{enumi}}
  \item $H^0(\PP^n,\OO_{\PP^n}(d))=S_d$ for all $d$.
  \item $H^i(\PP^n,\OO_{\PP^n}(d))=0$ for $0<i<n$ and for all $d$.
  \item The vector space $H^n(\PP^n,\OO_{\PP^n}(d))$ is isomorphic to the dual of the vector space $H^0(\PP^n,\OO_{\PP^n}(-d-n-1))$.
 \end{enumerate}
\end{theorem}

\begin{lemma}
Let $X=\PP^{a_1}\times\cdots\times\PP^{a_n}$, $0\leq p< \dim(X) -1$ and $k$ be a positive integer.
If  $\displaystyle{\sum_{i=1}^np_i<}0$ then $h^p(X,\OO_X (p_1,\cdots,p_{n})^{\oplus k}) = 0$ . 
\end{lemma}

\begin{proof}
As a consequence of K\"{u}nneth's formula we have\\
$\displaystyle{H^p(X,\OO_X(p_1,\cdots,p_{n}) )\cong \bigoplus
H^{q_1}(\PP^{a_1},\OO_{\PP^{a_1}}(p_1))\otimes H^{q_2}(\PP^{a_2},\OO_{\PP^{a_2}}(p_2))\otimes\cdots\otimes
H^{q_{n}}(\PP^{a_n},\OO_{\PP^{a_n}}(p_{n}))}$.\\
Now if $p_1+\cdots+p_n<0$ then either $H^{0}(\PP^{a_1},\OO_{\PP^{a_1}}(p_1))=0$ or $H^{0}(\PP^{a_2},\OO_{\PP^{a_2}}(p_2))=0$ or $\ldots$ or 
$H^{0}(\PP^{a_n},\OO_{\PP^{a_n}}(p_n))=0$.\\
Thus $H^1(X,\OO_X(p_1,\cdots,p_{n}))=0$ since it will contain one of the vanishing factors above.\\
We have $\displaystyle{\sum_{i=1}^np_i<}0 \Longleftrightarrow \displaystyle{p_1+\sum_{i=2}^np_i<}0$ and so if we consider summands with
$1\leq p_1$, $p_i\leq p-1$ for $i=2,\cdots,n$ when $2\leq p\leq \dim(X)-1$.\\
We conclude that if $\displaystyle{\sum_{i=1}^np_i<}0$ then $H^p(X,\OO_X (p_1,\cdots,p_{n})) = 0$ for $0\leq p< \dim(X) -1$.
Now for any $k$ be a positive integer, since $\OO_X (p_1,\cdots,p_{n})^{\otimes k}=\OO_X (kp_1,\cdots,kp_{n})$ we get then desired result.

\end{proof}

\begin{lemma}
Let $A$ and $B$ be vector bundles canonically pulled back from $A'$ on $\PP^n$ and $B'$ on $\PP^m$ then\\
$\displaystyle{H^q(\bigwedge^s(A\oplus B))=
\sum_{k_1+\cdots+k_s=q}\big\{\bigoplus_{i=1}^{s}(\sum_{j=0}^s\sum_{m=0}^{k_i}H^m(\wedge^j(A))\otimes(H^{k_i-m}(\wedge^{s-j}(B)))) \big\}}$.
\end{lemma}

\begin{proof}
The proof follows:
\begin{enumerate}
 \renewcommand{\theenumi}{\alph{enumi}}
  \item $\displaystyle{H^q(A_1\oplus\cdots\oplus A_s) = \bigoplus\left(H^{k_1}(A_1)\otimes H^{k_2}(A_2)\otimes\cdots \otimes H^{k_s}(A_s)\right)}$.\\
  \item $\displaystyle{H^q(A\otimes B)=\sum_{m=0}^qH^m(A)\otimes H^{q-m}(B)}$.\\
  \item $\displaystyle{\wedge^s(A\oplus B)=\sum_{j=0}^s\wedge^j(A)\otimes\wedge^{s-j}(B)}$.
\end{enumerate}

\end{proof}

\noindent The existence of monads on projective spaces was established by Fl\o{}ystad in \cite{3}, Main Theorem and generalized for a larger set of
projective varieties by Marchesi et al\cite{13}.

\begin{lemma}[\cite{13}, Theorem 2.4] Let $N\geq1$. There exists monads on $\PP^{N}$ whose maps are matrices of linear forms,
\[
\begin{CD}
0@>>>{\OO_{\PP^{N}}(-1)^{\oplus a}} @>>^{f}>{\OO^{\oplus b}_{\PP^{N}}} @>>^{g}>{\OO_{\PP^{N}}(1)^{\oplus c}} @>>>0\\
\end{CD}
\]
if and only if one of the following conditions holds\\
$(1) b\geq a+c$ and $b\geq 2c+N-1$ \\
$(2)b\geq a+c+N$.\\
If so there actually exists a monad with the map $f$ degenerating in expected codimension $b-a-c+1$.\\
If the cohomology of the monad is a vector bundle of rank less than $N$ then $N=2l+1$ is odd and the monad has the form
\[
\begin{CD}
0@>>>{\OO_{\PP^{2l+1}}(-1)^{\oplus a}} @>>^{f}>{\OO^{\oplus b}_{\PP^{2l+1}}} @>>^{g}>{\OO_{\PP^{2l+1}}(1)^{\oplus c}} @>>>0\\
\end{CD}
\] conversely for every $c,l\geq0$ there exists a monad as above whose cohomology is a vector bundle.
\end{lemma}

\section{Monads on a Cartesian product of odd dimension projective spaces}

\noindent The goal of this section is to construct monads on a Cartesian product of odd projective spaces $(\PP^1)^{l_1}\times(\PP^3)^{l_2}\times\cdots\times(\PP^{2n+1})^{l_m}$.
More specifically we generalize the results of Maingi \cite{12} by varying the ambient space. 
We prove that the kernel bundle is stable and thereafter we prove that the cohomology bundle $E$ associated to the monad on $X$ is simple.

\subsection{Monad construction via morphisms}

The following construction is a generalization of Construction 3.1 see \cite{12}.

\begin{construct}
Let $\psi : X = (\PP^1)^{l_1}\times(\PP^3)^{l_2}\times(\PP^5)^{l_3}\times\cdots\times(\PP^{2n+1})^{l_m}\longrightarrow \PP^{N=2n+1}$ be the Segre embedding which is defined as follows:\\
\\
$\psi([\alpha^1_{10}:\alpha^1_{11}]\cdots[\alpha^1_{l_10}:\alpha^1_{l_11}]
[\alpha^2_{10}:\alpha^2_{11}:\alpha^2_{12}:\alpha^2_{13}]\cdots[\alpha^2_{l_20}:\alpha^2_{l_21}:\alpha^2_{l_22}:\alpha^2_{l_23}]
[\alpha^3_{10}:\alpha^3_{11}:\alpha^3_{12}:\alpha^3_{13}:\alpha^3_{14}:\alpha^3_{15}]\cdots[\alpha^3_{l_30}:\alpha^3_{l_31}:\alpha^3_{l_32}:\alpha^3_{l_33}:\alpha^3_{l_34}:\alpha^3_{l_35}]\cdots
[\alpha^m_{10}:\alpha^m_{11}:\cdots:\alpha^m_{1(2n+1)}]\cdots[\alpha^m_{l_m0}:\alpha^m_{l_m1}:\cdots:\alpha^m_{l_m(2n+1)}])$
\[
\begin{CD}
@.@V^{injects}VV\\
@.[x_0:x_1:\cdots:x_{\nu}:y_0:y_2:\ldots:y_{\nu}]
\end{CD}\]
\\
First note that since we are taking $l_1,l_2,l_3,\cdots l_m$ copies of $\PP^1,\PP^3,\PP^5,\cdots,\PP^{2n+1}$ then we have \\
$N = 2^{l_1}\cdot4^{l_2}\cdot6^{l_3}\cdots(2n+2)^{l_m}-1$\\
$= 2^{l_1}\cdot4^{l_2}\cdot6^{l_3}\cdots(2n+2)^{l_m}-2+1$\\
$= 2\{2^{l_1-1}\cdot4^{l_2}\cdot6^{l_3}\cdots(2n+2)^{l_m}-1\}+1$\\
$=2{\nu}+1$\\
i.e. $N=2{\nu}+1$ where ${\nu}=\{2^{l_1-1}\cdot4^{l_2}\cdot6^{l_3}\cdots(2n+2)^{l_m}-1\}$.\\
 There exists a  linear monad of the form
\[
\begin{CD}
0@>>>\OO_{\PP^{2{\nu}+1}}(-1)^{\oplus k} @>>^{{A}}>{\OO^{\oplus2{\nu}+2k}_{\PP^{2{\nu}+1}}} @>>^{{B}}>\OO_{\PP^{2{\nu}+1}}(1)^{\oplus k} @>>>0\\
\end{CD}
\]
whose morphisms $A$ and $B$ that establish the monad are as given below
\[ B :=\left( \begin{array}{cccc|cccccccc}
x_0\cdots  & x_{\nu} &       &   &y_0 \cdots  & y_{\nu}\\
    &\ddots&\ddots &&\ddots&\ddots\\
    && x_0\cdots   x_{\nu} & & & y_0 \cdots  & y_{\nu}
\end{array} \right)
\]
 and
\[ A :=\left( \begin{array}{cccccccc}
-y_0\cdots  & -y_{\nu} \\
    	   &\ddots &\ddots\\
             &&-y_0 \cdots & -y_{\nu}\\
\hline
x_0 \cdots  & x_{\nu} \\
    	    &\ddots &\ddots\\
             && x_0\cdots & x_{\nu}\\
\end{array} \right)
\]
\\
We induce a monad on $X = (\PP^1)^{l_1}\times(\PP^3)^{l_2}\times(\PP^5)^{l_3}\times\cdots\times(\PP^{2n+1})^{l_m}$%
 \[\begin{CD}
M_\bullet: 0@>>>\OO_{X}(-1,\cdots,-1)^{\oplus{k}}@>>^{\overline{A}}>{\OO^{\oplus{2n}\oplus{2k}}_X} @>>^{\overline{B}}>\OO_{X}(1,\cdots,1)^{\oplus{k}}  @>>>0\\
\end{CD}\]
by giving the morphisms $\overline{A}$ and $\overline{B}$ with $\overline{B}\cdot\overline{A}=0$ and $\overline{A}$ and $\overline{B}$ are of maximal rank.
\\

From $A$ and $B$ whose entries are $x_0,\cdots,x_{\nu},y_0,\cdots,y_{\nu}$ the homogeneous coordinates on $\PP^{2{\nu}+1}$ we give the correspondence
for the the Segre embedding using the following table:\\
\[ \begin{array}{|c|c|}
\hline
homog. coord. ~~on ~~\PP^{2n+1} & representation ~homog. coord.~~ on ~~X\\
\hline
x_0 & \alpha^1_{10}\cdots\alpha^1_{l_10}\alpha^2_{10}\cdots\alpha^2_{l_20}\alpha^3_{10}\cdots\alpha^3_{l_30}\cdots\alpha^m_{10}\cdots\alpha^m_{l_m0} \\
x_1 & \alpha^1_{10}\cdots\alpha^1_{l_10}\alpha^2_{10}\cdots\alpha^2_{l_20}\alpha^3_{10}\cdots\alpha^3_{l_30}\cdots\alpha^m_{10}\cdots\alpha^m_{l_m1} \\
x_2 & \alpha^1_{10}\cdots\alpha^1_{l_10}\alpha^2_{10}\cdots\alpha^2_{l_20}\alpha^3_{10}\cdots\alpha^3_{l_30}\cdots\alpha^m_{10}\cdots\alpha^m_{l_m2} \\
\vdots&\vdots\\
x_{\nu-1} &  \alpha^1_{10}\cdots\alpha^1_{l_10}\alpha^2_{10}\cdots\alpha^2_{l_20}\alpha^3_{10}\cdots\alpha^3_{l_30}\cdots\alpha^m_{10}\cdots\alpha^m_{l_mn-1} \\
x_\nu &  \alpha^1_{10}\cdots\alpha^1_{l_10}\alpha^2_{10}\cdots\alpha^2_{l_20}\alpha^3_{10}\cdots\alpha^3_{l_30}\cdots\alpha^m_{10}\cdots\alpha^m_{l_mn} \\
y_0 &  \alpha^1_{10}\cdots\alpha^1_{l_10}\alpha^2_{10}\cdots\alpha^2_{l_20}\alpha^3_{10}\cdots\alpha^3_{l_30}\cdots\alpha^m_{10}\cdots\alpha^m_{l_mn+1} \\
\vdots&\vdots\\
y_{\nu-1} & \alpha^1_{11}\cdots\alpha^1_{l_11}\alpha^2_{13}\cdots\alpha^2_{l_23}\alpha^3_{15}\cdots\alpha^3_{l_35}\cdots\alpha^m_{1(2n+1)}\cdots\alpha^m_{l_m(2n)} \\
y_\nu &  \alpha^1_{11}\cdots\alpha^1_{l_11}\alpha^2_{13}\cdots\alpha^2_{l_23}\alpha^3_{15}\cdots\alpha^3_{l_35}\cdots\alpha^m_{1(2n+1)}\cdots\alpha^m_{l_m(2n+1)} \\
\hline
\end{array} 
\]
In following table we give a representation for the coordinates of $X$;
\[ \begin{array}{|c|c|}
\hline
homog. coord. ~~on ~~\PP^{2n+1} & homog. coord.~~ on ~~X\\
\hline
\rho_{10\cdots{l_1}010\cdots{l_2}010\cdots{l_3}0\cdots10\cdots{l_m}0} & \alpha^1_{10}\cdots\alpha^1_{l_10}\alpha^2_{10}\cdots\alpha^2_{l_20}\alpha^3_{10}\cdots\alpha^3_{l_30}\cdots\alpha^m_{10}\cdots\alpha^m_{l_m0} \\
\rho_{10\cdots{l_1}010\cdots{l_2}010\cdots{l_3}0\cdots10\cdots{l_m}1} & \alpha^1_{11}\cdots\alpha^1_{l_10}\alpha^2_{10}\cdots\alpha^2_{l_20}\alpha^3_{10}\cdots\alpha^3_{l_30}\cdots\alpha^m_{10}\cdots\alpha^m_{l_m1} \\
\rho_{10\cdots{l_1}010\cdots{l_2}010\cdots{l_3}0\cdots10\cdots{l_m}2} & \alpha^1_{12}\cdots\alpha^1_{l_10}\alpha^2_{10}\cdots\alpha^2_{l_20}\alpha^3_{10}\cdots\alpha^3_{l_30}\cdots\alpha^m_{10}\cdots\alpha^m_{l_m2} \\
\vdots&\vdots\\
\rho_{10\cdots{l_1}010\cdots{l_2}010\cdots{l_3}0\cdots10\cdots{l_m}n} & \alpha^1_{12}\cdots\alpha^1_{l_10}\alpha^2_{10}\cdots\alpha^2_{l_20}\alpha^3_{10}\cdots\alpha^3_{l_30}\cdots\alpha^m_{10}\cdots\alpha^m_{l_mn} \\
\rho_{10\cdots{l_1}010\cdots{l_2}010\cdots{l_3}0\cdots10\cdots{l_m}n+1} & \alpha^1_{12}\cdots\alpha^1_{l_10}\alpha^2_{10}\cdots\alpha^2_{l_20}\alpha^3_{10}\cdots\alpha^3_{l_30}\cdots\alpha^m_{10}\cdots\alpha^m_{l_mn+1} \\
\vdots&\vdots\\
\rho_{11\cdots{l_1}113\cdots{l_2}315\cdots{l_3}5\cdots1(2n+1)\cdots{l_m}(2n)} & \alpha^1_{11}\cdots\alpha^1_{l_11}\alpha^2_{13}\cdots\alpha^2_{l_23}\alpha^3_{15}\cdots\alpha^3_{l_35}\cdots\alpha^m_{1(2n+1)}\cdots\alpha^m_{l_m(2n)} \\
\rho_{11\cdots{l_1}113\cdots{l_2}315\cdots{l_3}5\cdots1(2n+1)\cdots{l_m}(2n+1)} & \alpha^1_{11}\cdots\alpha^1_{l_11}\alpha^2_{13}\cdots\alpha^2_{l_23}\alpha^3_{15}\cdots\alpha^3_{l_35}\cdots\alpha^m_{1(2n+1)}\cdots\alpha^m_{l_m(2n+1)} \\
\hline
\end{array} 
\]

Specifically we define two matrices $\overline{A}$ and $\overline{B}$ as follows\\
\[ \overline{B} =\left( \begin{array}{c|c}
B_1 & B_2
         \end{array} \right)\] and

\[ \overline{A}=\left( \begin{array}{cc}
A_1 \\ A_2
         \end{array} \right)\]
Where 
\[ \overline{B_1} :=\left[ \begin{array}{ccccccc}
\rho_{10\cdots{l_1}010\cdots{l_2}010\cdots{l_3}0\cdots10\cdots{l_m}0}\cdots  & \rho_{10\cdots{l_1}010\cdots{l_2}010\cdots{l_3}0\cdots10\cdots{l_m}n}  \\
    \ddots&\ddots \\
    & \rho_{10\cdots{l_1}010\cdots{l_2}010\cdots{l_3}0\cdots10\cdots{l_m}0}~~~ \cdots~~~  \rho_{10\cdots{l_1}010\cdots{l_2}010\cdots{l_3}0\cdots10\cdots{l_m}n} 
\end{array} \right]
\]
\\
\[ \overline{B_2} :=\left[ \begin{array}{ccccccc}
\rho_{10\cdots{l_1}010\cdots{l_2}010\cdots{l_3}0\cdots{l_m}n+1} \cdots  & \rho_{11\cdots{l_1}113\cdots{l_2}315\cdots{l_3}5\cdots{l_m}(2n+1)}\\
    \ddots&\ddots \\
    & \rho_{10\cdots{l_1}010\cdots{l_2}010\cdots{l_3}0\cdots{l_m}n+1}\cdots  & \rho_{11\cdots{l_1}113\cdots{l_2}315\cdots{l_3}5\cdots{l_m}(2n+1)}
\end{array} \right]
\]

and

\[ \overline{A_1} :=\left[ \begin{array}{ccccccc}
\rho_{10\cdots{l_1}010\cdots{l_2}010\cdots{l_3}0\cdots{l_m}n+1} \cdots  & \rho_{11\cdots{l_1}113\cdots{l_2}315\cdots{l_3}5\cdots{l_m}(2n+1)}\\
    \ddots&\ddots \\
    & \rho_{10\cdots{l_1}010\cdots{l_2}010\cdot{l_3}0\cdots{l_m}n+1}\cdots  & \rho_{11\cdots{l_1}113\cdots{l_2}315\cdots{l_3}5\cdots{l_m}(2n+1)}
\end{array} \right]
\]
\\
\[ \overline{A_2} :=\left[ \begin{array}{ccccccc}
\rho_{10\cdots{l_1}010\cdots{l_2}010\cdots{l_3}0\cdots10\cdots{l_m}0}\cdots  & \rho_{10\cdots{l_1}010\cdots{l_2}010\cdots{l_3}0\cdots10\cdots{l_m}n}  \\
    \ddots&\ddots \\
    & \rho_{10\cdots{l_1}010\cdots{l_2}010\cdots{l_3}0\cdots10\cdots{l_m}0}\cdots  \rho_{10\cdots{l_1}010\cdots{l_2}010\cdots{l_3}0\cdots10\cdots{l_m}n} 
\end{array} \right]
\]

We note that
\begin{enumerate}
 \item $\overline{B}\cdot \overline{A} = 0$ and
 \\
 \item  The matrices $\overline{B}$ and $\overline{A}$ have maximal rank\\
\end{enumerate}
Hence we get the desired monad,
\[\begin{CD}
M_\bullet: 0@>>>\OO_{X}(-1,\cdots,-1)^{\oplus{k}}@>>^{\overline{A}}>{\OO^{\oplus{2n}\oplus{2k}}_X} @>>^{\overline{B}}>\OO_{X}(1,\cdots,1)^{\oplus{k}}  @>>>0\\
\end{CD}\]
\end{construct}

\begin{theorem}
 Let $X = (\PP^1)^{l_1}\times(\PP^3)^{l_2}\times\cdots\times(\PP^{2n+1})^{l_m}$,  be a Cartesian product of $l_1$ copies 
of $\PP^1$, $l_2$ copies of $\PP^3$ $\cdots$ and $l_m$ copies of $\PP^{2n+1}$. There exists a monad of the form
 \[
\begin{CD}
M_\bullet: 0@>>>\OO_{X}(-1,\cdots,-1)^{\oplus{k}}@>>^{\overline{A}}>{\OO^{\oplus{2\nu}\oplus{2k}}_X} @>>^{\overline{B}}>\OO_{X}(1,\cdots,1)^{\oplus{k}}  @>>>0\\
\end{CD}
\]
where $l_1,\cdots,l_m,\nu,k$ are positive integers and $\nu=2^{l_1-1}4^{l_2}\cdots(2n+2)^{l_m}-1$.
\end{theorem}

\begin{proof} 
We have $a=c=k$, $b=2\nu+2k$ and $\nu=\dim X$,  $X = (\PP^1)^{l_1}\times(\PP^3)^{l_2}\times\cdots\times(\PP^{2n+1})^{l_m}$.\\
We show that conditions for Lemma 2.13 hold i.e. $b\geq a+c+N$, here $N=\nu$.\\
Now 
\begin{align*}
\begin{split}
b & = 2k+2\nu\\
  & = 2k+ 2[2^{l_1-1}4^{l_2}\cdots(2n+2)^{l_m}-1]\\
  & = 2k+ 2^{l_1}4^{l_2}\cdots(2n+2)^{l_m}-2\\
  & > 2k+ l_1+3l_2+5l_3+\cdots+(2n+1)l_m\\
  & = 2k + \dim X\\
  & = a+c+\nu
\end{split}
\end{align*}
Thus $b= 2k+2\nu\geq 2k + \dim X = a+c+\nu$, thus condition 2 holds.\\
Since $b>a+c+\nu>a+c+\nu-1$ and $b=2k+2\nu\geq 2k = a+c$ then Condition 1 holds.\\
From the above construction we get the explicit morphisms $\overline{A}$ and $\overline{B}$ and since the conditions of Lemma 2.13 hold the monad exists.

\end{proof}

\begin{lemma}
Let $T$ be a vector bundle on $X = (\PP^1)^{l_1}\times(\PP^3)^{l_2}\times\cdots\times(\PP^{2n+1})^{l_m}$ defined by the short exact sequence
\[\begin{CD}0@>>>T @>>>\OO_X^{2\nu+2k}@>>^{\overline{B}}>\OO_X(1,\cdots,1)^{\oplus k} @>>>0\end{CD}\]
then $T$ is stable for an ample line bundle $\mathscr{L} = \OO_X(1,\cdots,1)$
\end{lemma}

\begin{proof}
We show $H^0(X,(\bigwedge^q T)_{{\mathscr{L}}-norm}(g_{11},\cdots,g_{1l_1},g_{21}\cdots,g_{2l_2},\cdots g_{m1},\cdots,g_{ml_m})) = 0$ for all
$\displaystyle{\sum_{i=1}^{l_1}g_{1i}<0}$, $\displaystyle{\sum_{j=1}^{l_2}g_{2j}<0}$,$\cdots$, $\displaystyle{\sum_{k=1}^{l_m}g_{mk}<0}$ and $1\leq q\leq \rk(T)-1$.\\
\\
Consider the ample line bundle $\mathscr{L} = \OO_X(1,\cdots,1) = \OO(L)$. \\
From the short exact sequence \[\begin{CD}0@>>>T @>>>\OO_X^{2\nu+2k}@>>^{\overline{B}}>\OO_X(1,\cdots,1)^{\oplus k} @>>>0\end{CD}\] we get
\[c_1(T) = (-k,\cdots,-k)=-k(1,\cdots,1)\]
Since $L^{l}>0$, $\displaystyle{l=\sum_{i=1}^ml_m}$ then degree of $T$ is given by
$\deg_{\mathscr{L}}T=c_1(T)\cdot\mathscr{L}^{d-1}$\\
$\displaystyle{= -k(1,\cdots,1)\OO_X(g_{11}h_{11}+\cdots+g_{ml_m}h_{ml_m})^{l-1}}<0$\\ 

Since $\deg_{\mathscr{L}}T<0$, then $(\bigwedge^q T)_{\mathscr{L}-norm} = (\bigwedge^q T)$ and  it suffices by  Proposition 2.6, 
to prove that $h^0(\bigwedge^q T(g_{11},\cdots,g_{1l_1},g_{21},\cdots,g_{2l_2},\cdots,g_{m1},\cdots,g_{ml_m})) = 0$ with 
$\displaystyle{\sum_{i=1}^{l_1}g_{1i}<}0$, $\displaystyle{\sum_{j=1}^{l_2}g_{2j}<}0$,$\cdots$, $\displaystyle{\sum_{k=1}^{l_m}g_{mk}<}0$ for all $1\leq q\leq \rk(T)-1$.\\
\\
Next we twist the exact sequence 
\[\begin{CD}
0@>>>T @>>>{\OO_X^{2\nu+2k}} @>>>\OO_X(1,\cdots,1)^{\oplus k} @>>>0
\end{CD}\]
by $\OO_X (g_{11},\cdots,g_{1l_1},g_{21},\cdots,g_{2l_2},g_{m1},\cdots,g_{mk})$ we get the sequence,\\
\\
$0\lra T(g_{11},\cdots,g_{1l_1},g_{21},\cdots,g_{2l_2},g_{m1},\cdots,g_{mk})\lra\\
\lra\OO_X(g_{11},\cdots,g_{1l_1},g_{21},\cdots,g_{2l_2},g_{m1},\cdots,g_{mk})^{\oplus{2\nu+2k}}\lra\\
\lra\OO_X(1+g_{11},\cdots,1+g_{1l_1},1+g_{21},\cdots,1+g_{2l_2},1+g_{m1},\cdots,1+g_{mk})^{\oplus k}\lra0$\\
\\
and taking the exterior powers of the sequence by Proposition 2.7 we obtain\\
\\
$0\lra \bigwedge^q T(g_{11},\cdots,g_{1l_1},g_{21},\cdots,g_{2l_2},g_{m1},\cdots,g_{mk}) \lra\\
\lra \bigwedge^q (\OO_X(g_{11},\cdots,g_{1l_1},g_{21},\cdots,g_{2l_2},g_{m1},\cdots,g_{mk})^{\oplus{2\nu+2k}})\lra\\
\lra \bigwedge^{q-1}(\OO_X(1+2g_{11},\cdots,1+2g_{1l_1},1+2g_{21},\cdots,1+2g_{2l_2},1+2g_{m1},\cdots,1+2g_{mk})^{\oplus k})\cdots$\\
\\
Taking cohomology we have the injection:\\
$0\lra H^0(\bigwedge^{q}T(g_{11},\cdots,g_{1l_1},g_{21},\cdots,g_{2l_2},g_{m1},\cdots,g_{mk}))\\
\hookrightarrow H^0(\bigwedge^{q}(\OO_X(g_{11},\cdots,g_{1l_1},g_{21},\cdots,g_{2l_2},g_{m1},\cdots,g_{mk}))^{\oplus k}$\\
\\%
since $\displaystyle{\sum_{i=1}^{l_1}g_{1i}<}0$, $\displaystyle{\sum_{j=1}^{l_2}g_{2j}<}0$,$\cdots$, $\displaystyle{\sum_{k=1}^{l_m}g_{mk}<}0$ using Lemma 2.11 varying
the ambient space from $\PP^{a_1}\times\cdots\times\PP^{a_n}$ to $(\PP^1)^{l_1}\times(\PP^3)^{l_2}\times\cdots\times(\PP^{2n+1})^{l_m}$
then \[h^0(X,\bigwedge^{q}(\OO_X(g_{11},\cdots,g_{1l_1},g_{21},\cdots,g_{2l_2},g_{m1},\cdots,g_{mk})^{\oplus k}))= 0\]
thus it follows $h^0(\bigwedge^{q}T(g_{11},\cdots,g_{1l_1},g_{21},\cdots,g_{2l_2},g_{m1},\cdots,g_{mk}))=0$ and hence $T$ is stable.

\end{proof}

\begin{theorem} Let $X = (\PP^1)^{l_1}\times(\PP^3)^{l_2}\times\cdots\times(\PP^{2n+1})^{l_m}$, then the cohomology vector bundle $E$ associated to the monad 
\[
\begin{CD}
0@>>>{\OO_X(-1,\cdots,-1)^{\oplus k}} @>>^{f}>{\OO_X^{2\nu+2k}}@>>^{g}>\OO_X(1,\cdots,1)^{\oplus k} @>>>0
\end{CD}
\]
of rank $2\nu$ is simple.
\end{theorem}

\begin{proof}
The display of the monad is
\[
\begin{CD}
@.@.0@.0\\
@.@.@VVV@VVV\\
0@>>>{\OO_{X}(-1,\cdots,-1)^{\oplus k}} @>>>T=\ker(\overline{B})@>>>E=\ker(\overline{B})/\im(\overline{A})@>>>0\\
@.||@.@VVV@VVV\\
0@>>>{\OO_{X}(-1,\cdots,-1)^{\oplus k}} @>>^{\overline{A}}>{\OO^{\oplus2\nu+2k}_{X}}@>>>Q=\cok(\overline{A})@>>>0\\
@.@.@V^{\overline{B}}VV@VVV\\
@.@.{\OO_{X}(1,\cdots,1)^{\oplus k} }@={\OO_{X}(1,\cdots,1)^{\oplus k} }\\
@.@.@VVV@VVV\\
@.@.0@.0
\end{CD}
\]

\noindent To prove that $E$ is simple, we rely on the stability of the bundle $T$,  analysis of the display of the monad, dualizing, tensoring and 
twisting as follows;
\\
Take the dual of $\begin{CD}0@>>>{\OO_{X}(-1,\cdots,-1)^{\oplus k}} @>>>T@>>>E@>>>0\end{CD}$ and tensor it by  $E$ to obtain
\[
\begin{CD}
0@>>>E\otimes E^* @>>>E\otimes T^* @>>>E(1,\cdots,1)^k@>>>0\end{CD}
\]
and on taking cohomology it follows
\begin{equation}
h^0(X,E\otimes E^*) \leq h^0(X,E\otimes T^*)
\end{equation}
\\
We now dualize the short exact sequence $\begin{CD}0@>>>T @>>>\OO^{2\nu+2k}@>>>{\OO_{X}(1,\cdots,1)^{\oplus k}}@>>>0\end{CD}$
to get
\[\begin{CD}0@>>>\OO_X(-1,\cdots,-1)^{\oplus k} @>>>{\OO_X^{\oplus2\nu+2k}} @>>>T^* @>>>0\end{CD}\]
which on twisting by $\OO_X(-1,\cdots,-1)$, we get 
\[\begin{CD}0@>>>\OO_X(-2,\cdots,-2)^{\oplus k} @>>>{\OO_X(-1,\cdots,-1)^{\oplus2\nu+2k}} @>>>T^*(-1,\cdots,-1) @>>>0\end{CD}\]
taking cohomology and applying Lemma 2.11 we deduce 
\[H^0(X,T^*(-1,\cdots,-1)) = H^1(X,T^*(-1,\cdots,-1)) = 0\]
Finally tensor the short exact sequence  
\[\begin{CD}0@>>>{\OO_{X}(-1,\cdots,-1)^{\oplus k}} @>>>T@>>>E@>>>0\end{CD}\] by $T^*$ we obtain 
\[\begin{CD}0@>>>T^*(-1,\cdots,-1)^{\oplus k} @>>>T\otimes T^*@>>>E\otimes T^*@>>>0\end{CD}\] and on taking cohomology we obtain
\[\begin{CD}0@>>>H^0(T^*(-1,\cdots,-1)^{\oplus k}) @>>>H^0(T\otimes T^*)@>>>H^0(E\otimes T^*)@>>>0\end{CD}\]
and since $H^1(X,T^*(-1,\cdots,-1)^k=0$ for $k>1$ from the above lemma 
and from (1) above so we have 
\[1\leq h^0(X,T\otimes T^*)\leq h^0(X,E\otimes E^*) \leq h^0(X,E\otimes T^*)\leq1\]
It thus follows $ h^0(X,E\otimes E^*) = 1 $ and thus $E$ is simple.

\end{proof}

\section{Monads and bundles on $X = (\PP^{n})^2\times(\PP^{m})^2\times(\PP^{l})^2$ }

\vspace{0.5cm}

\noindent We now set up for monads on the multiprojective space $X = (\PP^{n})^2\times(\PP^{m})^2\times(\PP^{l})^2$.

\begin{lemma}
Let $\alpha,\beta,\gamma,m,n$ and $l$ be positive integers, given 6 matrices $f_{\alpha},f_{\beta},$ and $f_{\gamma}$ and $g_{\alpha},g_{\beta},$ and $g_{\gamma}$as shown;
\vspace{0.5cm}

\[ f_{\alpha} =\left[ \begin{array}{ccccc|ccccc}
&v^{\alpha}_{n} \cdots v^{\alpha+ak}_{0} & & -u^{\alpha}_{n} \cdots  -u^{\alpha+ak}_{0}\\
\adots&\adots &\adots&\adots \\
v^{\alpha}_{n} \cdots v^{\alpha+ak}_{0} & & -u^{\alpha}_{n} \cdots  -u^{\alpha+ak}_{0} \end{array} \right]_{k\times 2{(n+k)}}\]

\vspace{0.5cm}

\[ f_{\beta} =\left[ \begin{array}{ccccc|ccccc}
&x^{\beta}_{m} \cdots x^{\beta+bk}_{0} & & -w^{\beta}_{m} \cdots  -w^{\beta+bk}_{0}\\
\adots&\adots &\adots&\adots \\
x^{\beta}_{m} \cdots x^{\beta+bk}_{0} & & -w^{\beta}_{m} \cdots  -w^{\beta+bk}_{0} \end{array} \right]_{k\times {2(m+k)}}\]

\[ f_{\gamma} =\left[ \begin{array}{ccccc|ccccc}
&z^{\gamma}_{l} \cdots z^{\gamma+ck}_{0} & & -y^{\gamma}_{l} \cdots  -y^{\gamma+ck}_{0}\\
\adots&\adots &\adots&\adots \\
z^{\gamma}_{l} \cdots z^{\gamma+ck}_{0} & & -y^{\gamma}_{l} \cdots  -y^{\gamma+ck}_{0} \end{array} \right]_{k\times {2(l+k)}}\]

\vspace{0.5cm}

\[ g_{\alpha} =\left[\begin{array}{cccccc}
u^{\alpha}_{0}\\
\vdots &\ddots
 & u^{\alpha+ak}_{0}\\
u^{\alpha}_{n} &\ddots &\vdots\\
&& u^{\alpha+ak}_{n}\\
v^{\alpha}_{0}\\
\vdots &\ddots
 & v^{\alpha+ak}_{0}\\
v^{\alpha}_{n} &\ddots &\vdots\\
&& v^{\alpha+ak}_{n}
\end{array} \right]_{{2(n+k)}\times k}\] 

\vspace{0.5cm}

\[ g_{\beta} =\left[\begin{array}{cccccc}
w^{\beta}_{0}\\
\vdots &\ddots
 & w^{\beta+bk}_{0}\\
w^{\beta}_{m} &\ddots &\vdots\\
&& w^{\beta+bk}_{m}\\
x^{\beta}_{0}\\
\vdots &\ddots
 & x^{\beta+bk}_{0}\\
x^{\beta}_{m} &\ddots &\vdots\\
&& x^{\beta+bk}_{m}
\end{array} \right]_{{2(m+k)}\times k}\] 

\vspace{0.5cm}

\[ g_{\gamma} =\left[ \begin{array}{cccccc}
y^{\gamma}_{0}\\
\vdots &\ddots
 & y^{\gamma+ck}_{0}\\
y^{\gamma}_{l} &\ddots &\vdots\\
&& y^{\gamma+ck}_{l}\\
z^{\gamma}_{0}\\
\vdots &\ddots
 & z^{\gamma+ck}_{0}\\
z^{\gamma}_{l} &\ddots &\vdots\\
&& z^{\gamma+ck}_{l}
\end{array} \right]_{{2(l+k)}\times k}\] 

\vspace{0.5cm}

we define two matrices $f$ and $g$ as follows\\
\[ f =\left[\begin{array}{ccc} f_{\alpha} & f_{\beta} & f_{\gamma} \end{array} \right]\] and

\[ g =\left[\begin{array}{cc} g_{\alpha} \\ g_{\beta} \\ g_{\gamma} \end{array} \right]\]
then :\\
\begin{enumerate}
  \item $f\cdot g = 0$ and
  \item The matrices $f$ and $g$ have maximal rank
\end{enumerate}
\end{lemma}

\begin{proof}

\begin{enumerate}
   \item Now $f\cdot g =\left[\begin{array}{ccc} f_{\alpha} & f_{\beta} & f_{\gamma} \end{array} \right]\left[\begin{array}{cc} g_{\alpha} \\ g_{\beta} \\ g_{\gamma} \end{array} \right]$\\
   \[=\left[\begin{array}{ccc} f_{\alpha}g_{\alpha} & f_{\beta}g_{\beta}  & f_{\gamma}g_{\gamma} \end{array} \right]\]
   for $f_{\alpha} =\left[\begin{array}{c|c} \Huge{V}_{n}^{\alpha} & -\Huge{U}_{n}^{\alpha} \end{array} \right]$,
   $f_{\beta} =\left[\begin{array}{c|c} \Huge{X}_{m}^{\beta} & -\Huge{W}_{m}^{\beta} \end{array} \right]$,
   $f_{\gamma} =\left[\begin{array}{c|c} \Huge{Z}_{l}^{\gamma} & -\Huge{Y}_{l}^{\gamma} \end{array} \right]$,\\
   $g_{\alpha}=\left[\begin{array}{cc}~\Huge{U}_{n}^{\alpha} \\ \Huge{V}_{n}^{\alpha} \end{array} \right]$,
   $g_{\beta}=\left[\begin{array}{cc}~\Huge{W}_{m}^{\beta} \\ \Huge{X}_{m}^{\beta} \end{array} \right]$,
   $g_{\gamma}=\left[\begin{array}{cc}~\Huge{Y}_{l}^{\gamma} \\ \Huge{Z}_{l}^{\gamma} \end{array} \right]$.\\
   Since we have the following
  \begin{enumerate}\renewcommand{\theenumi}{\alph{enumi}}
   \item $\displaystyle{f_{\alpha}\cdot g_{\alpha}=\sum_{i=0}^n\sum_{j=0}^n\left(u^{\alpha}_iv^{\alpha}_j-u^{\alpha}_iv^{\alpha}_j\right)}$ and\\
   
   \item $\displaystyle{f_{\beta}\cdot g_{\beta}=\sum_{i=0}^m\sum_{j=0}^m\left(w^{\beta}_ix^{\beta}_j-w^{\beta}_ix^{\beta}_j\right)}$ and
   
   \item $\displaystyle{f_{\gamma}\cdot g_{\gamma}=\sum_{i=0}^l\sum_{j=0}^l\left(y^{\gamma}_iz^{\gamma}_j-y^{\gamma}_iz^{\gamma}_j\right)}$
  \end{enumerate}
 then it follows $f\cdot g $ is the zero matrix.
\\
 \item Notice that the rank of the two matrices drops if and only if all 
 $u^{\alpha}_{0},\cdots,u^{\alpha}_{n}$, $v^{\alpha}_{0},\cdots,v^{\alpha}_{n}$,
 $w^{\beta}_{0},\cdots,w^{\beta}_{m}$, $x^{\beta}_{0},\cdots,x^{\beta}_{m}$ and 
 $y^{\gamma}_{0},\cdots,y^{\gamma}_{l}$, $z^{\gamma}_{0},\cdots,z^{\gamma}_{l}$ are zeros and this is not possible in a projective space. Hence maximal rank.
\end{enumerate}
\end{proof}

\noindent Using the matrices given in the above lemma we are going to construct a monad.

\begin{theorem}
Let $\alpha,\beta,\gamma,l,m,n$ and $k$ be positive integers. Then there exists a linear monad on $X = (\PP^{n})^2\times(\PP^{m})^2\times(\PP^{l})^2$ of the form;
\[\begin{CD}0@>>>{\OO_X(-\alpha,-\alpha,-\beta,-\beta,-\gamma,-\gamma)^{\oplus k}} @>>^{f}>{\mathscr{G}_{\alpha}\oplus\mathscr{G}_{\beta}\oplus\mathscr{G}_{\gamma}}@>>^{g}>\OO_X(\alpha,\alpha,\beta,\beta,\gamma,\gamma)^{\oplus k} @>>>0\end{CD}\]
where 
\begin{align*}
\mathscr{G}_{\alpha}:=\OO_X(-\alpha,0,0,0,0,0)^{\oplus n+\oplus k}\oplus\OO_X(0,-\alpha,0,0,0,0)^{\oplus n+\oplus k}\\
\mathscr{G}_{\beta}:=\OO_X(0,0,-\beta,0,0,0)^{\oplus m+\oplus k}\oplus\OO_X(0,0,0,-\beta,0,0)^{\oplus m+\oplus k}\\
\mathscr{G}_{\gamma}:=\OO_X(0,0,0,0,-\gamma,0)^{\oplus l+\oplus k}\oplus\OO_X(0,0,0,0,0,-\gamma)^{\oplus l+\oplus k}
\end{align*}
\end{theorem}

\begin{proof}
The maps $f$ and $g$ in the monad are the matrices given in Lemma 3.4.\\
Notice that\\
$f\in$ Hom$(\OO_X(-\alpha,-\alpha,-\beta,-\beta,-\gamma,-\gamma)^{\oplus k},\mathscr{G}_{\alpha}\oplus\mathscr{G}_{\beta}\oplus\mathscr{G}_{\gamma})$ and \\
$g\in$ Hom$(\mathscr{G}_{\alpha}\oplus\mathscr{G}_{\beta}\oplus\mathscr{G}_{\gamma},\OO_X(\alpha,\alpha,\beta,\beta,\gamma,\gamma)^{\oplus k})$. \\
Hence by the above lemma they define the desired monad.
\end{proof}

\begin{lemma}
Let $K$ be the kernel bundle that sits in the short exact sequence
\[\begin{CD}0@>>>T @>>>\mathscr{G}_{\alpha}\oplus\mathscr{G}_{\beta}\oplus\mathscr{G}_{\gamma}@>>^{g}>\OO_X(1,\cdots,1)^{\oplus k} @>>>0\end{CD}\]
where 
\begin{align*}
\mathscr{G}_{\alpha}:=\OO_X(-\alpha,0,0,0,0,0)^{\oplus n+\oplus k}\oplus\OO_X(0,-\alpha,0,0,0,0)^{\oplus n+\oplus k}\\
\mathscr{G}_{\beta}:=\OO_X(0,0,-\beta,0,0,0)^{\oplus m+\oplus k}\oplus\OO_X(0,0,0,-\beta,0,0)^{\oplus m+\oplus k}\\
\mathscr{G}_{\gamma}:=\OO_X(0,0,0,0,-\gamma,0)^{\oplus l+\oplus k}\oplus\OO_X(0,0,0,0,0,-\gamma)^{\oplus l+\oplus k}
\end{align*}
\end{lemma}

\begin{proof}

We need to show that $H^0(X,\bigwedge^q T(p_1,p_2,p_3,p_4,p_5,p_6))=0$ for all $p_1+p_2+p_3+p_4+p_5+p_6<0$ and $1\leq q\leq \rk(T)-1$.\\
\\
Consider the ample line bundle $\mathscr{L} = \OO_X(\alpha,\alpha,\beta,\beta,\gamma,\gamma) = \OO(L)$. \\
Its class in 
$\pic(X)= \langle h_{1n},h_{2n},h_{1m},h_{2m},h_{1l},h_{2l}\rangle$ corresponds to the class\\
$1\cdot[g_1\times\PP^{n}]+1\cdot[\PP^{n}\times g_2]+1\cdot[g_3\times\PP^{m}]+1\cdot[\PP^{m}\times g_4]+1\cdot[g_5\times\PP^{l}]+1\cdot[\PP^{l}\times g_6]$ and
\\
Now from the display diagram of the monad we get \\ 
$c_1(K) = c_1(\mathscr{G}_{\alpha}\oplus\mathscr{G}_{\beta}\oplus\mathscr{G}_{\gamma}) - c_1(\OO_X(\alpha,\alpha,\beta,\beta,\gamma,\gamma)^{\oplus k})\\
       = (n+k)[(-\alpha,0,0,0,0,0)+(0,-\alpha,0,0,0,0)] + (m+k)[(0,0,-\beta,0,0,0)+(0,0,0,-\beta,0,0)] + (l+k)[(0,0,0,0,-\gamma,0)+(0,0,0,0,0,-\gamma)]- k(\alpha,\alpha,\beta,\beta,\gamma,\gamma) \\
       = (-n\alpha-2k\alpha,-n\alpha-2k\alpha,-m\beta-2k\beta,-m\beta-2k\beta,-l\gamma-2k\gamma,-l\gamma-2k\gamma) $\\
Since $L^{2(n+m+l)}>0$ the degree of $T$ is $\deg_{\mathscr{L}}T = c_1(T)\cdot\mathscr{L}^{d-1}$\\
\begin{align*}
\begin{split}
=-(n+m+l+6k)([g_1\times\PP^{n}]+[\PP^{n}\times g_2]+[g_3\times\PP^{m}]+[\PP^{m}\times g_4]+[g_5\times\PP^{l}]+[\PP^{l}\times g_6])\\
(1\cdot[g_1\times\PP^{n}]+1\cdot[\PP^{n}\times g_2]+1\cdot[g_3\times\PP^{m}]+1\cdot[\PP^{m}\times g_4]+1\cdot[g_5\times\PP^{l}]+1\cdot[\PP^{l}\times g_6])^{2n+2m+2l-1}\\
\end{split}
\end{align*}
$=-(n+m+l+6k)L^{2(n+m+l)}< 0$\\
\\
Since $\deg_{\mathscr{L}}T<0$, then $(\bigwedge^q T)_{\mathscr{L}-norm} = (\bigwedge^q T)$ and  it suffices by 
Lemma 2.7, to prove that $h^0(\bigwedge^q T(p_1,p_2,p_3,p_4,p_5,p_6)) = 0$ with $p_1+p_2+p_3+p_4+p_5+p_6<0$ and $1\leq q\leq \rk(T)-1$.\\
\\
First, twist the exact sequence 
\[\begin{CD}0@>>>T @>>>\mathscr{G}_{\alpha}\oplus\mathscr{G}_{\beta}\oplus\mathscr{G}_{\gamma}@>>^{g}>\OO_X(1,\cdots,1)^{\oplus k} @>>>0\end{CD}\]
by $\OO_X(p_1,p_2,p_3,p_4,p_5,p_6)$ we get,
\[
0\lra T(p_1,p_2,p_3,p_4,p_5,p_6)\lra\mathscr{\overline{G}}_{\alpha}\oplus\mathscr{\overline{G}}_{\beta}\oplus\mathscr{\overline{G}}_{\gamma}\lra\OO_X(1+p_1,1+p_2,1+p_3,1+p_4,1+p_5,1+p_6)^{\oplus k}\lra0\]
where 
\begin{align*}
\mathscr{\overline{G}}_{\alpha}:=\OO_X(p_1-\alpha,p_2,p_3,p_4,p_5,p_6)^{\oplus n+\oplus k}\oplus\OO_X(p_1,p_2-\alpha,p_3,p_4,p_5,p_6)^{\oplus n+\oplus k}\\
\mathscr{\overline{G}}_{\beta}:=\OO_X(p_1,p_2,p_3-\beta,p_4,p_5,p_6)^{\oplus m+\oplus k}\oplus\OO_X(p_1,p_2,p_3,p_4-\beta,p_5,p_6)^{\oplus m+\oplus k}\\
\mathscr{\overline{G}}_{\gamma}:=\OO_X(p_1,p_2,p_3,p_4,p_5-\gamma,p_6)^{\oplus l+\oplus k}\oplus\OO_X(p_1,p_2,p_3,p_4,p_5,p_6-\gamma)^{\oplus l+\oplus k}
\end{align*}

and taking the exterior powers of the sequence by Fact 2.8 we get
\[0\lra \bigwedge^q T(p_1,p_2,p_3,p_4,p_5,p_6) \lra \bigwedge^q (\mathscr{\overline{G}}_{\alpha}\oplus\mathscr{\overline{G}}_{\beta}\oplus\mathscr{\overline{G}}_{\gamma)}\lra \cdots\]
Taking cohomology we have the injection:
\[0\lra H^0(X,\bigwedge^{q}T(p_1,p_2,p_3,p_4,p_5,p_6))\hookrightarrow H^0(X,\bigwedge^q (\mathscr{\overline{G}}_{\alpha}\oplus\mathscr{\overline{G}}_{\beta}\oplus\mathscr{\overline{G}}_{\gamma}))\]
From Theorem 2.10 and Lemma 2.11 we have $H^0(X,\bigwedge^q (\mathscr{\overline{G}}_{\alpha}\oplus\mathscr{\overline{G}}_{\beta}\oplus\mathscr{\overline{G}}_{\gamma}))=0$.\\
\\
$\Longrightarrow$ $H^0(X,\bigwedge^{q}T(p_1,p_2,p_3,p_4,p_5,p_6)) =  H^0(X,\bigwedge^q (\mathscr{\overline{G}}_1\oplus\cdots\oplus\mathscr{\overline{G}}_n)=0$ hence $T$ is stable.

\end{proof}

\begin{lemma} The cohomology bundle $E$ associated to the monad
\[\begin{CD}0@>>>{\OO_X(-\alpha,-\alpha,-\beta,-\beta,-\gamma,-\gamma)^{\oplus k}} @>>^{f}>{\mathscr{G}_{\alpha}\oplus\mathscr{G}_{\beta}\oplus\mathscr{G}_{\gamma}}@>>^{g}>\OO_X(\alpha,\alpha,\beta,\beta,\gamma,\gamma)^{\oplus k} @>>>0\end{CD}\]
of rank $2(n+m+l+2k)$ is simple where $X = (\PP^{n})^2\times(\PP^{m})^2\times(\PP^{l})^2$.
\end{lemma}

\begin{proof}
The display of the monad is
\[
\begin{CD}
@.@.0@.0\\
@.@.@VVV@VVV\\
0@>>>{\OO_X(-\alpha,-\alpha,-\beta,-\beta,-\gamma,-\gamma)^{\oplus k}} @>>>T=\ker g@>>>E\rightarrow0\\
@.||@.@VVV@VVV\\
0@>>>{\OO_X(-\alpha,-\alpha,-\beta,-\beta,-\gamma,-\gamma)^{\oplus k}} @>>^{f}>{\mathscr{G}_n\oplus\cdots\oplus\mathscr{G}_m}@>>>Q=\cok f\rightarrow0\\
@.@.@V^{g}VV@VVV\\
@.@.{\OO_X(\alpha,\alpha,\beta,\beta,\gamma,\gamma)^{\oplus k}}@={\OO_X(\alpha,\alpha,\beta,\beta,\gamma,\gamma)^{\oplus k}}\\
@.@.@VVV@VVV\\
@.@.0@.0
\end{CD}
\]

\noindent Since $T$ the kernel of the map $g$ is stable from the above Lemma 4.3, we prove that the cohomology bundle $E=\ker g/\im f$  is simple.\\
\\
The first step is to take the dual of the short exact sequence 
\[\begin{CD}
0@>>>\OO_X(-\alpha,-\alpha,-\beta,-\beta,-\gamma,-\gamma)^{\oplus k} @>>>T@>>>E @>>>0
\end{CD}\]
to get
\[
\begin{CD}
0@>>>E^* @>>>T^* @>>>\OO_X(\alpha,\alpha,\beta,\beta,\gamma,\gamma)^{\oplus k}@>>>0.
\end{CD}
\]
Tensoring by $E$ we get\\
\[
\begin{CD}
0@>>>E\otimes E^* @>>>E\otimes T^* @>>>E(\alpha,\alpha,\beta,\beta,\gamma,\gamma)^k@>>>0.
\end{CD}
\]
Now taking cohomology gives:
\[\begin{CD}
0@>>>H^0(X,E\otimes E^*) @>>>H^0(X,E\otimes T^*) @>>>H^0(E(\alpha,\alpha,\beta,\beta,\gamma,\gamma)^{\oplus k})@>>>\cdots
\end{CD}\]
\\
which implies that 
\begin{equation}
h^0(X,E\otimes E^*) \leq h^0(X,E\otimes T^*)
\end{equation}
\\
Now we dualize the short exact sequence
\[\begin{CD}
0@>>>T @>>>\mathscr{G}_{\alpha}\oplus\mathscr{G}_{\beta}\oplus\mathscr{G}_{\gamma} @>>>\OO_X(\alpha,\alpha,\beta,\beta,\gamma,\gamma)^{\oplus k} @>>>0
\end{CD}\]
\\
to get
\[\begin{CD}
0@>>>\OO_X(-\alpha,-\alpha,-\beta,-\beta,-\gamma,-\gamma)^{\oplus k} @>>>{\mathscr{G}_{\alpha}\oplus\mathscr{G}_{\beta}\oplus\mathscr{G}_{\gamma} } @>>>T^* @>>>0
\end{CD}\]
\\
Twisting the short exact sequence above by $\OO_X(-\alpha,-\alpha,-\beta,-\beta,-\gamma,-\gamma)$ yields
\[\begin{CD}
0@>>>\OO_X(-2\alpha,-2\alpha,-2\beta,-2\beta,-2\gamma,-2\gamma)^{\oplus k} @>>>{\mathscr{G}'_{\alpha}\oplus\mathscr{G}'_{\beta}\oplus\mathscr{G}'_{\gamma} }\\ @>>>T^*(-\alpha,-\alpha,-\beta,-\beta,-\gamma,-\gamma) @>>>0
\end{CD}\]
where 
\begin{align*}
\mathscr{G}'_{\alpha}:=\OO_X(-2\alpha,-\alpha,-\beta,-\beta,-\gamma,-\gamma)^{\oplus n+\oplus k}\oplus\OO_X(-\alpha,-2\alpha,-\beta,-\beta,-\gamma,-\gamma)^{\oplus n+\oplus k}\\
\mathscr{G}'_{\beta}:=\OO_X(-\alpha,-\alpha,-2\beta,-\beta,-\gamma,-\gamma)^{\oplus m+\oplus k}\oplus\OO_X(-\alpha,-\alpha,-\beta,-2\beta,-\gamma,-\gamma)^{\oplus m+\oplus k}\\
\mathscr{G}'_{\gamma}:=\OO_X(-\alpha,-\alpha,-\beta,-\beta,-2\gamma,-\gamma)^{\oplus l+\oplus k}\oplus\OO_X(-\alpha,-\alpha,-\beta,-\beta,-\gamma,-2\gamma)^{\oplus l+\oplus k}
\end{align*}%
next on taking cohomology one gets
\[\begin{CD}
0\lra H^0(\OO_X(-2\alpha,-2\alpha,-2\beta,-2\beta,-2\gamma,-2\gamma)^k) \lra H^0(\mathscr{G'}_{\alpha})\oplus H^0(\mathscr{G'}_{\beta})\oplus H^0(\mathscr{G'}_{\gamma})\lra\\\lra H^0(T^*(-\alpha,-\alpha,-\beta,-\beta,-\gamma,-\gamma) )\lra\\
\lra H^1(\OO_X(-2\alpha,-2\alpha,-2\beta,-2\beta,-2\gamma,-2\gamma)^k) \lra H^1(\mathscr{G'}_{\alpha})\oplus H^1(\mathscr{G'}_{\beta})\oplus H^0(\mathscr{G'}_{\gamma})\lra\\\lra H^1(T^*(-\alpha,-\alpha,-\beta,-\beta,-\gamma,-\gamma) )\lra\\
\lra H^2(X,\OO_X(-2\alpha,-2\alpha,-2\beta,-2\beta,-2\gamma,-2\gamma)^k) \lra H^2(\mathscr{G'}_{\alpha})\oplus H^2(\mathscr{G'}_{\beta})\oplus H^2(\mathscr{G'}_{\gamma})\lra\\\lra H^2(T^*(-\alpha,-\alpha,-\beta,-\beta,-\gamma,-\gamma) )\lra\cdots
\end{CD}
\]
As a consequence of from Theorem 2.10 and Lemma 2.11 we deduce  that \\
$H^0(X,T^*(-\alpha,-\alpha,-\beta,-\beta,-\gamma,-\gamma)) = 0$ and $H^1(X,T^*(-\alpha,-\alpha,-\beta,-\beta,-\gamma,-\gamma)) = 0$ 
\\
Lastly, tensor the short exact sequence
\[
\begin{CD}
0@>>>\OO(-\alpha,-\alpha,-\beta,-\beta,-\gamma,-\gamma)^{\oplus k} @>>>T @>>> E@>>>0\\
\end{CD}
\]
by $T^*$ to get
\[
\begin{CD}
0@>>>T^*(-\alpha,-\alpha,-\beta,-\beta,-\gamma,-\gamma)^k @>>>T\otimes T^* @>>> E\otimes T^*@>>>0\\
\end{CD}
\]
and taking cohomology we have
\\
\[
\begin{CD}
0@>>>H^0(X,T^*(-\alpha,-\alpha,-\beta,-\beta,-\gamma,-\gamma)^k) @>>>H^0(X,T\otimes T^*) @>>> H^0(X,E\otimes T^*)@>>>\\
@>>>H^1(X,T^*(-\alpha,-\alpha,-\beta,-\beta,-\gamma,-\gamma)^k)@>>>\cdots
\end{CD}
\]
\\
But since  $H^0(X,T^*(-\alpha,-\alpha,-\beta,-\beta,-\gamma,-\gamma)) = H^1(X,T^*(-\alpha,-\alpha,-\beta,-\beta,-\gamma,-\gamma)) = 0$ from above then  it follows $H^1(X,T^*(-\alpha,-\alpha,-\beta,-\beta,-\gamma,-\gamma)^k)=0$ for $k>1$.\\
\\
so we have 
\\
\[
\begin{CD}
0@>>>H^0(X,T^*(-\alpha,-\alpha,-\beta,-\beta,-\gamma,-\gamma)^{k}) @>>>H^0(X,T\otimes T^*) @>>> H^0(X,E\otimes T^*)@>>>0
\end{CD}
\]
\\
This implies that 
\begin{equation}
h^0(X,T\otimes T^*) \leq h^0(X,E\otimes T^*)
\end{equation}
\\
Since $T$ is stable then it follows that it is simple which implies $h^0(X,T\otimes T^*)=1$.\\
\\
From $(3)$ and $(4)$ and putting these together we have;\\
\[1\leq h^0(X,E\otimes E^*) \leq h^0(X,E\otimes T^*) = h^0(X,T\otimes T^*) = 1\]\\
\\
We have $ h^0(X,E\otimes E^*) = 1 $ and therefore $E$ is simple.

\end{proof}

\begin{theorem}
Let $X = (\PP^{n})^2\times(\PP^{m})^2\times(\PP^{l})^2$ and $\mathscr{L}=\OO_X(\alpha,\alpha,\beta,\beta,\gamma,\gamma)$ and ample line bundle,
then the monad 
\[\begin{CD}0@>>>{\OO_X(-\alpha,-\alpha,-\beta,-\beta,-\gamma,-\gamma)^{\oplus k}} @>>^{f}>{\mathscr{G}_{\alpha}\oplus\mathscr{G}_{\beta}\oplus\mathscr{G}_{\gamma}}@>>^{g}>\OO_X(\alpha,\alpha,\beta,\beta,\gamma,\gamma)^{\oplus k} @>>>0\end{CD}\]
where 
\begin{align*}
\mathscr{G}_{\alpha}:=\OO_X(-\alpha,0,0,0,0,0)^{\oplus n+\oplus k}\oplus\OO_X(0,-\alpha,0,0,0,0)^{\oplus n+\oplus k}\\
\mathscr{G}_{\beta}:=\OO_X(0,0,-\beta,0,0,0)^{\oplus m+\oplus k}\oplus\OO_X(0,0,0,-\beta,0,0)^{\oplus m+\oplus k}\\
\mathscr{G}_{\gamma}:=\OO_X(0,0,0,0,-\gamma,0)^{\oplus l+\oplus k}\oplus\OO_X(0,0,0,0,0,-\gamma)^{\oplus l+\oplus k}
\end{align*} has the properties:
\begin{enumerate}\renewcommand{\theenumi}{\alph{enumi}}
 \item The kernel bundle, $T=\ker(g)$ is $\mathscr{L}$-stable.
 
 \item The cohomology bundle $E$ of $\rk(E)=2(n+m+l+2k)$ is simple.

\end{enumerate}

\end{theorem}

\begin{proof}
(a) Follows from Lemma 4.3 and (b) follows from Lemma 4.4.
\end{proof}

\vspace{1cm}

\noindent \textbf{Data Availability statement}
My manuscript has no associate data.

\vspace{1cm}

\noindent \textbf{Conflict of interest}
On behalf of all authors, the corresponding author states that there is no conflict of interest.

\end{document}